\pdfoutput=1

\documentclass[10pt]{article}

\usepackage{geometry}
\usepackage[pdftex]{graphicx}
\usepackage{epstopdf}
\usepackage[pdftex,bookmarks=false,colorlinks=true,pdfstartview=FitBV,linkcolor=blue,citecolor=blue,urlcolor=blue]{hyperref}
\usepackage{amsmath} 
\usepackage{amssymb} 
\usepackage{amsthm}
\usepackage[noend]{algpseudocode}
\usepackage{algorithm}
\usepackage{graphicx}
\usepackage{float}
\usepackage{setspace}
\usepackage{subfigure}
\usepackage{color}
\usepackage{cases}
\usepackage{cite}
\usepackage{ mathrsfs }
\usepackage{comment} 
\usepackage{tabularx}
\usepackage{verbatim}

\newcommand{\diff}[1]{\text{d}{#1}}

\newtheorem{theorem}{Theorem}[section]
\newtheorem{lemma}[theorem]{Lemma}
\newtheorem{proposition}[theorem]{Proposition}

\newtheorem{remark}{Remark}
\newtheorem{definition}[theorem]{Definition}

\title{Covariance Steering for Discrete-Time Linear-Quadratic\\ Stochastic Dynamic Games}

\author{
  Venkata Ramana Makkapati%
  \thanks{School of Aerospace Engineering, Georgia Institute of Technology, 		Atlanta, GA 30332-0150, USA. Email: \texttt{\small \{mvramana, tsiotras\}@gatech.edu}}
  \and%
  Tanmay Rajpurohit%
  \thanks{Genpact Innovation Center, Palo Alto, CA 94304, USA. Email: \texttt{\small tanmay.rajpurohit@genpact.digital}}
  \and%
  Kazuhide Okamoto%
  \thanks{Zoox, Foster City, CA 94404, USA.}
  \and%
  Panagiotis Tsiotras\footnotemark[1]%
}

\begin{document}

\maketitle

\begin{abstract}                          
This paper addresses the problem of steering a discrete-time linear dynamical system from an initial Gaussian distribution to a final distribution in a game-theoretic setting. 
One of the two players strives to minimize a quadratic payoff, while at the same time tries to meet a given mean and covariance constraint at the final time-step.
The other player maximizes the same payoff, but it is assumed to be \emph{indifferent} to the terminal constraint.
At first, the unconstrained version of the game is examined, and the necessary conditions for the existence of a saddle point are obtained.
We then show that obtaining a solution for the \emph{one-sided} constrained dynamic game is not guaranteed, and subsequently the players' best responses are analyzed.
Finally, we propose to numerically solve the problem of steering the distribution under adversarial scenarios using the Jacobi iteration method.
The problem of guiding a missile during the endgame is chosen to analyze the proposed approach.
A numerical simulation corresponding to the case where the terminal distribution is not achieved is also included, and discuss the necessary conditions to meet the terminal constraint.
\end{abstract}


\section{Introduction}
	
	Stochastic games, introduced by Shapley in 1953, deal with instances where a stochastic process is jointly controlled by two players, a \emph{controller} and a \emph{stopper}, along with an underlying payoff function that is common to both players~\cite{shapley1953stochastic}. 
	The stopper tries to maximize the payoff function, while the controller strives to minimize it.
	The current work addresses a class of linear-quadratic (LQ) stochastic dynamic games in discrete-time with finite-time horizon. 
	It is assumed that the players have perfect measurement of the state at each time instant and that the initial state is sampled from a given Gaussian distribution.
	First, the problem of steering the covariance in an LQ game setting without any constraints is analyzed, and the associated saddle point equilibrium is identified. 
	Subsequently, the problem of steering the initial distribution to a specified terminal distribution (which is also Gaussian) under adversarial situations, which can be categorized as a \emph{general constrained game} (GCG) \cite{Altman2009}, is considered.
	This \emph{constrained covariance steering game} (CCSG) is relevant in the context of stochastic pursuit-evasion and has applications in spacecraft rendezvous \cite{Repperger1972}, collision avoidance\cite{jungsu2018}, and understanding behavioral patterns in nature \cite{free2018}.
	
	\subsection{Related Work}
	
	Prior work on stochastic dynamic games considered continuous-time linear systems with players having noisy measurements \cite{Speyer1967, Ho1968, Leondes1979, Bagchi1981, Willman1969, CASTANON1976}.
	In one of the early works, Speyer discussed the outcomes of a linear stochastic differential game in a defense setting with a missile and a radar \cite{Speyer1967}. 
	The radar's objective is to minimize the uncertainty in the current state of the missile using a filter, and the missile strives to maximize the same, while also trying to arrive at a target point with certain accuracy. 
	Further, the missile has linear dynamics and adds additive noise to its control in order to confuse its adversary. 
	Subsequently, zero-sum LQ differential game theory with noise-corrupted measurements was applied to study the pursuit-evasion setting involving a homing missile and an evading aircraft \cite{Leondes1979}.
	Quasi-linearization was used to solve the nonlinear two-point boundary-value problem with closed-form solutions.
	At the same time, Casta{\~n}{\'o}n and Athans derived feedback strategies for a two-person Stackelberg game with quadratic performance cost and linear dynamics \cite{CASTANON1976}.
	Kumar and van Schuppen considered different observations and costs for the two players \cite{Kumar1980}. 
	They obtained strategies for the players assuming that one of the players had a ``spy" creating an instance of asymmetric information.
	Bley and Stear studied discrete stochastic dynamic games in $\mathbb{R}$ \cite{BLEY1971}.
	Yavin analyzed various pursuit-evasion problems in the stochastic setting and provided sufficient conditions for the players' optimal strategies \cite{YAVIN1984,YAVIN1987}.
	Bernhard and Colomb provided a solution for the rabbit and the hunter game in a partial information setting, using dynamic programming \cite{Bernhard1988}.
	
	In recent years, stochastic games with different information structures, where the state transition and observation equations are linear, have been extensively studied \cite{Gupta2014, Basar2014, Swarup2003, Clemens2017, Rhodes1969, Chong1971}. 
	Rajpurohit et al. studied a two-player stochastic differential game with a nonlinear dynamical model over an infinite time horizon \cite{rajpurohit2017}.
	Additional results and applications for stochastic differential games can be found in Ref. \cite{SDG_book}.
	
	Owing to the fact that a Gaussian distribution can be fully defined using its first two moments, the problem discussed in this paper can be decomposed into mean and covariance steering problems \cite{okamoto2018Optimal, okamoto2019}. The mean steering problem is essentially a deterministic dynamic game. The necessary and sufficient conditions for the existence of a solution to the discrete-time LQ dynamic game was provided by Pachter and Pham, along with a closed-form solution \cite{Pachter2010}.
	
	The idea of covariance steering has its genesis in the 1980s.
	First introduced by Hotz and Skelton \cite{hotz1985covariance}, the problem of infinite-horizon covariance assignment for continuous and discrete-time systems has been analyzed by various researchers \cite{iwasaki1992quadratic, xu1992improved, grigoriadis1997minimum, Collins1987}.
	The finite-horizon equivalent of the problem in continuous-time was investigated only recently by Chen et al. \cite{chen2016optimalI, chen2016optimalII, chen2018optimalIII}, where it was shown that the related solutions have theoretical connections to the Schr{\"o}dinger bridges \cite{Schrodinger1931} and the optimal mass transport problems \cite{Kantorovich1942}.
	The solution to the problem of covariance steering in finite time is also of great importance to entry, descent, and landing problems \cite{ridderhof2018uncertainty}.
	Previous works have shown that these solutions can be seen in the context of LQG with a particular set of weights which can be solved in terms of LMIs \cite{halder2016finite, goldshtein2017finite, bakolas2016optimalACC, bakolas2016optimalCDC, bakolas2018finite}.
	
	\subsection{Contributions}
	
	The contributions of this work are as follows.
	\begin{enumerate}
		\item A novel LQ formulation for driving a Gaussian to a given terminal distribution under an adversarial setting is introduced. The adversary is assumed to be indifferent to the controller's terminal constraint which is unique to the literature on covariance steering.
		
		\item It is shown that the proposed game theoretic formulation can be decomposed into two independent games, mean steering and covariance steering games, which makes the problem tractable.
		
		\item The existence of equilibrium solutions is discussed for both unconstrained and constrained versions of the games.
		
		\item A condition in terms of \emph{relative controllability} is identified in the mean steering game with controller constraints for discrete systems.
		
		\item A simple Jacobi procedure for finding saddle points is introduced to solve the constrained covariance steering game, assuming a linear feedback control structure.
		
		\item The missile endgame guidance scenario is revisited, while assuming a process noise in the system, to demonstrate the proposed approach.
	\end{enumerate}
	
	At this point, it is worth mentioning that the attitude of a player towards its opponent's constraint influences the outcome of the GCG \cite{Altman2009}.
	Instances where the players are indifferent to couple constraints of their opponents can be found in mobile networks and finance \cite{Altman2009, perlaza2011quality}. 
	For example, in the situation where there are multiple mobile carriers competing to maximize the received power in a series of time slots, the networks are also subjected to a minimum expected throughput.
    Note that in a given time slot, only one carrier is successful, and the overall success of a network depends on the actions of all other mobiles, indicating coupled constraints with players being indifferent to their opponents' constraints.	
	In the case where a player's main goal is to prevent the opponent from meeting its constraint, his attitude is to be understood as being aggressive. 
	Analyzing the scenario where the stopper has an aggressive attitude towards the controller's constraint is beyond the scope of this work.
	
	The contents of the paper are as follows. 
	Section \ref{sec:prelims} establishes the notation and introduces the mathematical preliminaries, including some necessary definitions and existing results.
	The problem statement dealt in this paper is also included in Section \ref{sec:prelims}.
	In Section \ref{sec:separate}, the Gaussian steering problem under adversarial settings is separated into its corresponding mean and covariance steering problems, which are subsequently analyzed in Sections \ref{sec:mean} and \ref{sec:covariance}, respectively.
	Section \ref{sec:sims} presents the numerical simulations, and analyzes covariance steering in the context of missile endgame guidance problem.
	Finally, Section \ref{sec:conclude} concludes the paper.


\section{Mathematical Preliminaries}
	\label{sec:prelims}
	
	The notation used in this paper is as follows.
	The rank of a matrix $A$ is denoted as $\text{rank}[A]$, and $e_i$ denotes the elementary vector, where the operation $e^\top_iAe_j$ retains the $(i,j)$ element of matrix $A$.
	$A^{-1}$ denotes the inverse of matrix $A$.
	$I_n$ and $0_{n \times m}$ denote an identity matrix of size $n\times n$, and a zero matrix of size $n \times m$, respectively (the subscripts will be omitted when obvious from the context).
	The trace and determinant of a square matrix are denoted as $\text{tr}(\cdot)$ and $\text{det}(\cdot)$, respectively.
	The positive definiteness of a symmetric matrix $R$ is denoted as $R \succ 0$, and positive semi-definiteness is denoted as $R \succeq 0$.
	A random variable $x$ with normal distribution is denoted as $x \sim \mathcal{N}(\mu,\Sigma)$, where $\mu$ is its mean, and $\Sigma$ is its covariance matrix. 
	Finally, $\mathbb{E}[\cdot]$ denotes the expectation of a random variable. 
	
	Consider the following discrete-time linear stochastic system
	\begin{align}
	x_{k+1} = A_k x_k + B_k u_k + C_k v_k + D_k w_k, \label{eq:ini_system}
	\end{align}
	where $k = 0,~1,\dots,N-1$ is the time-step. At the $k^{\text{th}}$ time-step, $x_k \in \mathbb{R}^n$ denotes the state, $u_k \in \mathbb{R}^m$ is the controller input, $v_k \in \mathbb{R}^\ell$ is the stopper input, and $w_k \in \mathbb{R}^r$ is a zero-mean white Gaussian noise with unit covariance, i.e.
	\begin{align}
	\mathbb{E}[w_k] = 0, \quad \mathbb{E}[w_{k_1}w_{k_2}^\top] = \begin{cases}
	I_r, &\text{if } k_1 = k_2,\\
	0, &\text{otherwise}.
	\end{cases}
	\end{align}
	In addition, it is assumed that
	\begin{align}
	\mathbb{E}[x_{k_1}w_{k_2}^\top] = 0, \quad 0 \leq k_1 \leq k_2 \leq N.
	\end{align}
	The initial state $x_0$ is distributed according to $x_0 \sim \mathcal{N}(\mu_0,\Sigma_0)$,
	where $\mu_0 \in \mathbb{R}^{n}$ is the initial state mean, and $\Sigma_0 \in \mathbb{R}^{n \times n}$ is the initial state covariance, with $\Sigma_0 \succeq 0$. The payoff function is
	\begin{align}
	J(u_0,\dots,u_{N-1},v_0,\dots,v_{N-1}) = \mathbb{E}\left[\sum_{k=0}^{N-1} \left(x^\top_kQ_kx_k + u^\top_kR_ku_k - v^\top_kS_kv_k\right)\right]. \label{eq:ini_cost}
	\end{align}
	It is assumed that $Q_k \succeq 0$  for all $k = 0,\dots,N$, and $R_k,S_k \succ 0$ for all $k = 0,\dots,N-1$. The set of control inputs $\{u_0,\dots,u_{N-1}\}$ is chosen by one player to minimize the payoff function (\ref{eq:ini_cost}), and the control inputs $\{v_0,\dots,v_{N-1}\}$, are chosen by the adversary to maximize (\ref{eq:ini_cost}).
	
	Using the notation introduced in \cite{goldshtein2017finite}, the system dynamics in (\ref{eq:ini_system}) can be alternatively expressed as 
	\begin{align}
	x_k = \bar{A}_kx_0 + \bar{B}_kU_k + \bar{C}_kV_k + \bar{D}_kW_k, 	\label{eq:inter_sys}
	\end{align}
	where $U_k = [u_0,~u_1,\dots,u_{k-1}]^\top$, $V_k = [v_0,~v_1,\dots,v_{k-1}]^\top$, $W_k = [w_0,~w_1,\dots,w_{k-1}]^\top$ are the augmented control and noise profiles. 
	Furthermore, with the augmented state vector $X = [x_1,~x_2,\dots,x_N]^\top$, the system dynamics (\ref{eq:ini_system}) can be rewritten as
	\begin{align}
	X = \mathcal{A}x_0 + \mathcal{B}U + \mathcal{C}V + \mathcal{D}W, \label{eq:modi_sys}
	\end{align}
	where $U = U_{N_1}$, $V = V_{N_1}$, and $W = W_{N_1}$. The definitions of the big matrices ($A_k$, $\mathcal{A}, \dots$) can be found in \cite{goldshtein2017finite}.
Note that $\mathbb{E}[x_0x_0^\top] = \Sigma_0 + \mu_0\mu_0^\top,$ $\mathbb{E}[x_0W^\top] = 0,$ $\mathbb{E}[WW^\top] = I.$
Consequently, the payoff function in (\ref{eq:ini_cost}) can be expressed as
	\begin{align}
	J(U,V) = \mathbb{E}[X^\top\bar{Q}X + U^\top\bar{R}U - V^\top\bar{S}V], \label{eq:modi_cost}
	\end{align}
	where $\bar{Q} = \text{blkdiag}(Q_0,\dots,Q_{N-1},0) \in \mathbb{R}^{(N+1)n \times (N+1)n}$, $\bar{R} = \text{blkdiag}(R_0,R_1,\dots,R_{N-1}) \in \mathbb{R}^{Nm \times Nm}$, and $\bar{S} = \text{blkdiag}(S_0,S_1,\dots,S_{N-1}) \in \mathbb{R}^{N\ell \times N\ell}$. Also, since $Q_k \succeq 0$ for all $k = 0,\dots,N$, and $R_k,S_k \succ 0$ for all $k = 0,\dots,N-1$, it follows that $\bar{Q} \succeq 0$ and $\bar{R},\bar{S} \succ 0$.
The mean and the covariance of the initial state $x_0$ can be written in terms of $X$ as
	\begin{subequations}\label{eq:boun_start}
		\begin{align}
		\mu_0 &= E_0\mathbb{E}[X], \\ 
		\Sigma_0 &= E_0(\mathbb{E}[XX^\top] - \mathbb{E}[X]\mathbb{E}[X]^\top)E^\top_0,
		\end{align}
	\end{subequations}
	where $E_0 \triangleq [I_{n},0,\dots,0] \in \mathbb{R}^{n\times(N+1)n}$.

	\vspace{0.1in}
	
	\begin{definition}
		The \emph{upper game} is a scheme in which the stopper chooses $V$ based on the information it has on the control $U$, and the \emph{upper value} is defined by
		\begin{align}
		\mathcal{V}^+ = \underset{U \in \mathbb{R}^{Nm}}{\inf}~\underset{V \in \mathbb{R}^{N\ell}}{\sup}~J(U,V).
		\end{align}
		Similarly, the \emph{lower game} is a scheme in which the controller chooses $U$ based on the information it has on the control $V$, and the \emph{lower value} is defined by
		\begin{align}
		\mathcal{V}^- = \underset{V \in \mathbb{R}^{N\ell}}{\sup}~\underset{U \in \mathbb{R}^{Nm}}{\inf}~J(U,V).
		\end{align}
	\end{definition}
	
	\vspace{0.1in}
	
	It is well known that, in general $\mathcal{V}^- \leq \mathcal{V}^+$. If the \emph{Isaacs minimax condition} holds, then $\mathcal{V}^- = \mathcal{V}^+$, and the corresponding set of control actions $(U^*,V^*)$ is called the equilibrium solution or saddle point \cite{fleming1989existence}. 
	The unconstrained Gaussian steering problems to be addressed in this paper can now be stated as follows.
	
	\vspace{0.1in}
	
	\textbf{Problem 1.} Find the saddle point ($U^*,V^*$) for the \emph{unconstrained dynamic game} (UDG), described by the payoff function (\ref{eq:modi_cost}), the system (\ref{eq:modi_sys}), and the initial conditions (\ref{eq:boun_start}). 
	
	\vspace{0.1in}
	
	In this paper, as mentioned earlier, we propose to analyze the one-sided contrained dynamic game.
	To this end, let 
	\begin{align}
	E_NX = x_N \sim \mathcal{N}(\mu_N,\Sigma_N), \label{eq:boun_final}
	\end{align}
	where $E_N \triangleq [0,\dots,0,I_n] \in \mathbb{R}^{n\times(N+1)n}$, be terminal state that the controller strives to achieve at the final time-step.
	Note that it is only the controller who is concerned about meeting the terminal condition (\ref{eq:boun_final}), and hence (\ref{eq:boun_final}) is a one-sided constraint. 
	It is assumed that the stopper is aware of the controller's terminal constraint however, it is indifferent to this constraint, and it is solely interested in maximizing the payoff (\ref{eq:modi_cost}).
	Furthermore, since the terminal constraint (\ref{eq:boun_final}) is dependent on the control inputs of both players, and it is a one-sided constraint, the problem of interest can be categorized as a GCG \cite{Altman2009}.
	
	\vspace{0.1in}
	
	\begin{remark}
		The terminal condition (\ref{eq:boun_final}) can be used to enforce probabilistic capture in the case of a two-player pursuit-evasion game with $\mu_N = 0$, when (\ref{eq:ini_system}) represents the relative motion between the pursuer and the evader.
	\end{remark}

	\vspace{0.1in}
	
	We will now formally define the saddle point in the one-sided constrained dynamic game using the corresponding upper and lower values. For a given stopper action $V$, let $\mathcal{U}(V)$ denotes the set of controllers $U \in \mathbb{R}^{Nm}$ that drive the system to terminal Gaussian distribution in (\ref{eq:boun_final}), and let $\mathcal{R} \triangleq \bigcup_{V \in \mathbb{R}^{N\ell}}~\mathcal{U}(V) \subseteq \mathbb{R}^{Nm}$.
	
	\vspace{0.1in}
	
	\begin{definition}\label{def:constr_values}
		The \emph{constrained upper value} is defined by
		\begin{align}
		\mathcal{V}_c^+ = \underset{U \in \mathbb{R}^{Nm}}{\inf}~\underset{V \in \mathbb{R}^{N\ell}}{\sup}~J(U,V),
		\end{align}
		and the \emph{constrained lower value} is defined by
		\begin{align}
		\mathcal{V}_c^- = \underset{V \in \mathbb{R}^{N\ell}}{\sup}~\underset{U \in \mathcal{R}}{\inf}~J(U,V).
		\end{align}
	\end{definition}
	
	\vspace{0.1in}
	
	The existence of the constrained upper and lower values requires that the controller meets the terminal constraint in (\ref{eq:boun_final}) in the corresponding upper and lower games. Given the system dynamics (\ref{eq:modi_sys}) and the initial conditions (\ref{eq:boun_start}), note that for some $V$, there may not exist a controller such that the terminal condition (\ref{eq:boun_final}) can be met i.e., $\mathcal{U}(V) = \emptyset$. Consequently, there may not exist a constrained upper (or lower) value for the constrained game. Finally, a saddle point in the constrained game can be defined as ($U_c^*,V_c^*$) for which the $\mathcal{V}_c^+$ and $\mathcal{V}_c^-$ exist, and are equal. 
	
	\vspace{0.1in}
	
	\textbf{Problem 2.} Find the necessary conditions such that the controller can drive the system to the final state, while the stopper tries to maximize the payoff function (\ref{eq:modi_cost}), given the system dynamics (\ref{eq:modi_sys}) and the initial conditions (\ref{eq:boun_start}). 
	Furthermore, find the optimal control inputs for both players. 
	Hereafter, this problem will be referred to as the \emph{constrained dynamic game} (CDG).


\section{Separation of Mean and Covariance Control Problems}
	\label{sec:separate}
	
	In  \cite{okamoto2018Optimal}, it was demonstrated that the mean and the covariance evolutions of the system can be separated. Subsequently, by separating the cost, independent controllers that drive the mean and the covariance were derived. A similar approach is followed here by first observing the fact that
	\begin{align}
	\mu_k \triangleq \mathbb{E}[x_k] = \bar{A}_k\mu_0 + \bar{B}_k\bar{U}_k + \bar{C}_k\bar{V}_k, \label{eq:int_muevol}
	\end{align}
	where $\bar{U}_k = \mathbb{E}[U_k]$ and $\bar{V}_k = \mathbb{E}[V_k]$. By defining 
	\begin{align}
	\tilde{x}_k \triangleq x_k - \mu_k, \quad \tilde{U}_k \triangleq U_k - \bar{U}_k, \quad \tilde{V}_k \triangleq V_k - \bar{V}_k, \label{eq:decomposition}
	\end{align}
	and using (\ref{eq:inter_sys}), the following equation holds for $\tilde{x}_k$.
	\begin{align}
	\tilde{x}_k = \bar{A}_k\tilde{x}_0 + \bar{B}_k\tilde{U}_k + \bar{C}_k\tilde{V}_k + \bar{D}_kW_k. \label{eq:xtilde}
	\end{align}
	Subsequently,
	\begin{align}
	\Sigma_k &\triangleq \mathbb{E}[\tilde{x}_k\tilde{x}_k^\top] \nonumber \\
	&= \mathbb{E}\left[\left(\bar{A}_k\tilde{x}_0 + \bar{B}_k\tilde{U}_k + \bar{C}_k\tilde{V}_k + \bar{D}_kW_k\right) \left(\bar{A}_k\tilde{x}_0 + \bar{B}_k\tilde{U}_k + \bar{C}_k\tilde{V}_k + \bar{D}_kW_k\right)^\top\right].  \label{eq:int_covevo}
	\end{align}
	It can be observed that the mean evolution in (\ref{eq:int_muevol}) depends only on $\bar{U}_k$, $\bar{V}_k$, while the evolution of $\tilde{x}$ in (\ref{eq:xtilde}), and the covariance evolution in (\ref{eq:int_covevo}) depend only on $\tilde{U}_k$, $\tilde{V}_k$, and the noise profile $W_k$. Consequently, from (\ref{eq:modi_sys}) and (\ref{eq:int_muevol}), it follows that
	\begin{align}
	\bar{X} \triangleq \mathbb{E}[X] = \mathcal{A}\mu_0 + \mathcal{B}\bar{U} + \mathcal{C}\bar{V},
	\end{align}
	and from (\ref{eq:xtilde}),
	\begin{align}
	\tilde{X} \triangleq X - \mathbb{E}[X] = \mathcal{A}\tilde{x}_0 + \mathcal{B}\tilde{U} + \mathcal{C}\tilde{V} + \mathcal{D}W. \label{eq:Xtilde}
	\end{align}
	The objective function (\ref{eq:modi_cost}) can be further rewritten as
	\begin{align}
	J(U,V) &= \mathbb{E}[X^\top\bar{Q}X + U^\top\bar{R}U - V^\top\bar{S}V] \nonumber \\
	&= \text{tr}(\bar{Q}\mathbb{E}[\tilde{X} \tilde{X}^\top]) + \bar{X}^\top \bar{Q}\bar{X} + \text{tr}(\bar{R}\mathbb{E}[\tilde{U} \tilde{U}^\top]) + \bar{U}^\top \bar{R}\bar{U} - \text{tr}(\bar{S}\mathbb{E}[\tilde{V} \tilde{V}^\top]) - \bar{V}^\top \bar{S}\bar{V} \nonumber \\
	&= J_\mu(\bar{U},\bar{V}) + J_\Sigma(\tilde{U},\tilde{V}), \label{eq:cost_decompose}
	\end{align}
	where
	\begin{align}
	J_\mu(\bar{U},\bar{V}) = \bar{X}^\top \bar{Q}\bar{X} + \bar{U}^\top \bar{R}\bar{U} - \bar{V}^\top \bar{S}\bar{V}, \label{eq:mean_cost}
	\end{align}
	and
	\begin{align}
	J_\Sigma(\tilde{U},\tilde{V}) = \text{tr}(\bar{Q}\mathbb{E}[\tilde{X} \tilde{X}^\top]) + \text{tr}(\bar{R}\mathbb{E}[\tilde{U} \tilde{U}^\top]) - \text{tr}(\bar{S}\mathbb{E}[\tilde{V} \tilde{V}^\top]). \label{eq:cov_cost}
	\end{align}
	
	\vspace{0.1in}
	
	\begin{proposition} \label{prop:sep_meacov}
		For the UDG, the saddle point controls $(U^*,V^*)$ that solve the problem (if they exist) are given by $U^* = \bar{U}^* + \tilde{U}^*$ and $V^* = \bar{V}^* + \tilde{V}^*$, where $(\bar{U}^*,\bar{V}^*)$ solves the \emph{unconstrained mean steering game} 
		\begin{align}
		\text{(UMSG)}\begin{cases}
		\text{Payoff function: }J_\mu(\bar{U},\bar{V}), \\
		\text{where } \bar{X} = \mathcal{A}\mu_0 + \mathcal{B}\bar{U} + \mathcal{C}\bar{V},
		\end{cases} \label{eq:mean_game}
		\end{align}
		and $(\tilde{U}^*,\tilde{V}^*)$ solves the \emph{unconstrained covariance steering game}
		\begin{align}
		\text{(UCSG)}\begin{cases}
		\text{Payoff function: }J_\Sigma(\tilde{U},\tilde{V}), \\
		\text{where } \tilde{X} = \mathcal{A}\tilde{x}_0 + \mathcal{B}\tilde{U} + \mathcal{C}\tilde{V} + \mathcal{D}W.
		\end{cases} \label{eq:cov_game}
		\end{align}
	\end{proposition}
	
	\vspace{0.1in}
	
	\begin{proof}
		From (\ref{eq:decomposition}), (\ref{eq:cost_decompose}), it can observed that the mean payoff $J_\mu(\bar{U},\bar{V})$ in (\ref{eq:mean_cost}) is driven by the mean control actions $\bar{U}$ and $\bar{V}$ independently while the covariance payoff $J_\Sigma(\tilde{U},\tilde{V})$ in (\ref{eq:mean_cost}) is driven by the covariance control actions $\tilde{U}$ and $\tilde{V}$.
		As a result, the UDG in terms of $(U,V)$ is equivalent to two separate dynamic games in terms of $(\bar{U},\bar{V})$ and $(\tilde{U},\tilde{V})$ with payoff functions (\ref{eq:mean_cost}) and (\ref{eq:cov_cost}), respectively, leading to the result.
		
	\end{proof}

	\vspace{0.1in}
	
	\begin{proposition}\label{prop:constr_game_decompose}
		Proposition \ref{prop:sep_meacov} applies to the CDG as well, with $U^*_c = \bar{U}^*_c + \tilde{U}^*_c$, $V^*_c = \bar{V}^*_c + \tilde{V}^*_c$, where $(\bar{U}_c^*,\bar{V}_c^*)$ solves the \emph{constrained mean steering game} (CMSG)
		\begin{subequations}
			\begin{numcases}{}
			\text{Payoff function: }J_\mu(\bar{U},\bar{V}), \\
			\text{where } \bar{X} = \mathcal{A}\mu_0 + \mathcal{B}\bar{U} + \mathcal{C}\bar{V},\nonumber \\
			\text{Controller constraint:} \nonumber \\
			\mu_N = E_N\bar{X} = \bar{A}_N\mu_0 + \bar{B}_N\bar{U} + \bar{C}_N \bar{V}, \label{eq:org_meancont}
			\end{numcases}\label{eq:const_mean_game}
		\end{subequations}
		and $(\tilde{U}_c^*,\tilde{V}_c^*)$ solves the \emph{constrained covariance steering game} (CCSG) with
		\begin{subequations}
			\begin{numcases}{}
			\text{Payoff function: }J_\Sigma(\tilde{U},\tilde{V}), \\
			\text{where } \tilde{X} = \mathcal{A}\tilde{x}_0 + \mathcal{B}\tilde{U} + \mathcal{C}\tilde{V} + \mathcal{D}W, \nonumber \\
			\text{Controller constraint: } \nonumber \\
			\Sigma_N = E_N \left(\mathbb{E}[XX^\top] - \mathbb{E}[X]\mathbb{E}[X]^\top\right) E_N^\top, \label{eq:org_covcont}
			\end{numcases}\label{eq:const_cov_game}
		\end{subequations}
		where the constraints (\ref{eq:org_meancont}) and (\ref{eq:org_covcont}), as stated earlier, are of concern only for the controller.
	\end{proposition}
	
	\vspace{0.1in}
	
	\begin{proof}
	The proof is similar to the one given for Proposition \ref{prop:sep_meacov}
		From (\ref{eq:decomposition}), (\ref{eq:cost_decompose}), it can observed that the mean payoff $J_\mu(\bar{U},\bar{V})$ in (\ref{eq:mean_cost}) is driven by mean control actions $\bar{U}$ and $\bar{V}$, while the covariance payoff $J_\Sigma(\tilde{U},\tilde{V})$ in (\ref{eq:mean_cost}) is driven by covariance control actions $\tilde{U}$ and $\tilde{V}$, independently.
		Furthermore, the mean and the covariance control actions address the constraints (\ref{eq:org_meancont}) and (\ref{eq:org_covcont}) too in an independent fashion.
		As a result, the CDG in terms of $(U,V)$ is equivalent to two separate dynamic games in (\ref{eq:const_mean_game}) and (\ref{eq:const_cov_game}) in terms of $(\bar{U},\bar{V})$ and $(\tilde{U},\tilde{V})$, respectively, leading to the result.
	\end{proof}
	
	\vspace{0.1in}
	
	Note that non-existence of saddle point in either CMSG or CCSG or both, implies non-existence of saddle point in CDG. 
	For the analysis of mean steering game in the following section, we introduce the set $\bar{\mathcal{R}} $.
	For a given stopper action $\bar{V}$ in CMSG, let $\bar{\mathcal{U}}(\bar{V})$ denotes the set of mean controllers $\bar{U} \in \mathbb{R}^{Nm}$ that satisfies the constraint in (\ref{eq:org_meancont}), and let $\bar{\mathcal{R}} \triangleq \bigcup_{\bar{V} \in \mathbb{R}^{N\ell}}~\bar{\mathcal{U}}(\bar{V}) \subseteq \mathbb{R}^{Nm}$.


\section{Mean Steering Game}
	\label{sec:mean}
	
	The solution to the UMSG is given in the following proposition.
	
	\vspace{0.1in}
	
	\begin{proposition} \label{prop:umsg}
		Assume that
		\begin{align}
		\bar{S} - \mathcal{C}^\top \bar{Q} \mathcal{C} \succ 0, \label{eq:mean_assump}
		\end{align}
		then the saddle point $(\bar{U}^*,\bar{V}^*)$ that solves the UMSG (\ref{eq:mean_game}) is given by
		\begin{align}
		\left[\begin{array}{c}
		\bar{U}^* \\ \bar{V}^* 
		\end{array}\right] = -\left[\begin{array}{cc}
		\mathcal{B}^\top\bar{Q}\mathcal{B} + \bar{R} & \mathcal{B}^\top\bar{Q}\mathcal{C} \\
		\mathcal{C}^\top\bar{Q}\mathcal{B} & \mathcal{C}^\top\bar{Q}\mathcal{C} - \bar{S}
		\end{array}\right]^{-1}\left[\begin{array}{c}
		\mathcal{B}^\top\bar{Q}\mathcal{A} \\ \mathcal{C}^\top\bar{Q}\mathcal{A}
		\end{array}\right] \mu_0 \label{eq:mean_sol}
		\end{align}	
		and this solution is unique.
	\end{proposition}
	
	\vspace{0.1in}
	
	\begin{proof}
		The payoff function (\ref{eq:mean_cost}) can be expressed as
		\begin{align}
		J_\mu(\bar{U},\bar{V}) &= (\mathcal{A}\mu_0 + \mathcal{B}\bar{U} + \mathcal{C}\bar{V})^\top \bar{Q}(\mathcal{A}\mu_0 + \mathcal{B}\bar{U} + \mathcal{C}\bar{V}) + \bar{U}^\top \bar{R}\bar{U} - \bar{V}^\top \bar{S}\bar{V}.
		\end{align}
		The first-order necessary conditions \cite{Basar1999} for a saddle point yield
		\begin{subequations}
			\begin{align}
			\nabla_{\bar{U}}J_\mu &= (\mathcal{B}^\top\bar{Q}\mathcal{B} + \bar{R})\bar{U} + \mathcal{B}^\top\bar{Q}\mathcal{C}\bar{V} + \mathcal{B}^\top\bar{Q}\mathcal{A}\mu_0 = 0, \label{eq:mean_grad1}\\
			\nabla_{\bar{V}}J_\mu &= (\mathcal{C}^\top\bar{Q}\mathcal{C} - \bar{S})\bar{V} + \mathcal{C}^\top\bar{Q}\mathcal{B}\bar{U} + \mathcal{C}^\top\bar{Q}\mathcal{A}\mu_0 = 0. \label{eq:mean_grad2}
			\end{align}
		\end{subequations}
		The above two equations can be expressed as
		\begin{align}
		\left[\begin{array}{cc}
		\mathcal{B}^\top\bar{Q}\mathcal{B} + \bar{R} & \mathcal{B}^\top\bar{Q}\mathcal{C} \\
		\mathcal{C}^\top\bar{Q}\mathcal{B} & \mathcal{C}^\top\bar{Q}\mathcal{C} - \bar{S}
		\end{array}\right]\left[\begin{array}{c}
		\bar{U}^* \\ \bar{V}^* 
		\end{array}\right] = -\left[\begin{array}{c}
		\mathcal{B}^\top\bar{Q}\mathcal{A} \\ \mathcal{C}^\top\bar{Q}\mathcal{A}
		\end{array}\right] \mu_0, \label{eq:meansol_tfeq}
		\end{align}
		Let
		\begin{align}
		\mathcal{T}_m = \left[\begin{array}{cc}
		\mathcal{B}^\top\bar{Q}\mathcal{B} + \bar{R} & \mathcal{B}^\top\bar{Q}\mathcal{C} \\
		\mathcal{C}^\top\bar{Q}\mathcal{B} & \mathcal{C}^\top\bar{Q}\mathcal{C} - \bar{S}
		\end{array}\right],
		\end{align}
		and from (\ref{eq:mean_assump}), $\mathcal{B}^\top\bar{Q}\mathcal{C}(\mathcal{C}^\top\bar{Q}\mathcal{C} - \bar{S})^{-1}\mathcal{C}^\top\bar{Q}\mathcal{B} \prec 0$.
		As a result, $\mathcal{B}^\top\bar{Q}\mathcal{B} + \bar{R} - \mathcal{B}^\top\bar{Q}\mathcal{C}(\mathcal{C}^\top\bar{Q}\mathcal{C} - \bar{S})^{-1}\mathcal{C}^\top\bar{Q}\mathcal{B} \succ 0$.
		Therefore, $\text{det}(\mathcal{T}_m) = \text{det}(\mathcal{C}^\top\bar{Q}\mathcal{C} - \bar{S})\text{det}(\mathcal{B}^\top\bar{Q}\mathcal{B} + \bar{R} - \mathcal{B}^\top\bar{Q}\mathcal{C}(\mathcal{C}^\top\bar{Q}\mathcal{C} - \bar{S})^{-1}\mathcal{C}^\top\bar{Q}\mathcal{B}) \neq 0$, and $\mathcal{T}_m$ is invertible.
		Equation (\ref{eq:mean_sol}) then follows immediately from (\ref{eq:meansol_tfeq}).
		From (\ref{eq:mean_assump}), the second order derivatives yield
		\begin{subequations}
			\begin{align}
			\nabla_{\bar{U}\bar{U}}J_\mu &= \mathcal{B}^\top\bar{Q}\mathcal{B} + \bar{R} \succ 0, \\
			\nabla_{\bar{V}\bar{V}}J_\mu &= \mathcal{C}^\top\bar{Q}\mathcal{C} - \bar{S} \prec 0.
			\end{align}
			Therefore, the payoff function is convex in $\bar{U}$, and concave in $\bar{V}$. Hence $(\bar{U}^*,\bar{V}^*)$ is the only saddle point that solves the given dynamic game \cite{Basar1999}.
		\end{subequations}
	\end{proof}
	
	\vspace{0.1in}
	
	Next, we analyze the CMSG.
	As this is a constrained zero-sum game, we obtain the following inequality.
	A similar result can be found in Ref. \cite{Altman2009} (Theorem III.1).
	
	\vspace{0.1in}
	
	\begin{lemma}
		Assuming that the UMSG (\ref{eq:mean_game}) has a saddle point equilibrium (Proposition \ref{prop:umsg}), the CMSG (\ref{eq:const_mean_game}) satisfies
		\begin{align}
		\underset{\bar{U} \in \mathbb{R}^{Nm}}{\inf}~\underset{\bar{V} \in \mathbb{R}^{N\ell}}{\sup}~J_\mu(\bar{U},\bar{V}) \leq \underset{\bar{V} \in \mathbb{R}^{N\ell}}{\sup}~\underset{\bar{U} \in \bar{\mathcal{R}}}{\inf}~J_\mu(\bar{U},\bar{V}). \label{eq:infsup_supinf}
		\end{align} \label{lemma:mean_const}
	\end{lemma}
	
	\vspace{0.1in}
	
	\begin{proof} Given that the UMSG has a saddle point equilibrium, it follows that
		\begin{align}
		\underset{\bar{U} \in \mathbb{R}^{Nm}}{\inf}~\underset{\bar{V} \in \mathbb{R}^{N\ell}}{\sup}~J_\mu(\bar{U},\bar{V}) = \underset{\bar{V} \in \mathbb{R}^{N\ell}}{\sup}~\underset{\bar{U} \in \mathbb{R}^{Nm}}{\inf}~J_\mu(\bar{U},\bar{V}). \label{eq:infsupeq}
		\end{align}
		Since $\bar{\mathcal{R}} \subseteq \mathbb{R}^{Nm}$,
		\begin{align}
		\underset{\bar{U} \in \mathbb{R}^{Nm}}{\inf}~J_\mu(\bar{U},\bar{V}) \leq \underset{\bar{U} \in \bar{\mathcal{R}}}{\inf}~J_\mu(\bar{U},\bar{V}).
		\end{align}
		Hence,
		\begin{align}
		\underset{\bar{V} \in \mathbb{R}^{N\ell}}{\sup}~\underset{\bar{U} \in \mathbb{R}^{Nm}}{\inf}~J_\mu(\bar{U},\bar{V}) \leq \underset{\bar{V} \in \mathbb{R}^{N\ell}}{\sup}~\underset{\bar{U} \in \bar{\mathcal{R}}}{\inf}~J_\mu(\bar{U},\bar{V}),
		\end{align}
		and from (\ref{eq:infsupeq}), the result follows.
	\end{proof}
	
	\vspace{0.1in}
	
	As a result, a pure-strategy equilibrium might not exist for the CMSG, and only players' best responses can be obtained \cite{Altman2009}.
	To this end, the constrained upper and lower games for the CMSG problem can be examined. 
	As stated in Definition \ref{def:constr_values}, in the constrained lower game, the stopper has to choose its input first, while the controller has the advantage of obtaining the stopper input, and then choosing his best response accordingly. 
	
	\vspace{0.1in}
	
	\begin{lemma}
		Assuming that the discrete-time linear dynamical system (\ref{eq:ini_system}) is controllable for $C_k = 0$ and $D_k = 0$ (i.e., rank$[\bar{B}_N] = n$), the controller's feasible set (the set of controllers for which the constraint (\ref{eq:org_meancont}) is met given the stopper input) is non-empty for any $\bar{V} \in \mathbb{R}^{N\ell}$.
	\end{lemma}
	
	\vspace{0.1in}
	
	\begin{proof} For a given $\bar{V} \in \mathbb{R}^{N\ell}$, the mean constraint (\ref{eq:org_meancont}) can be rewritten as 
		\begin{align}
		\bar{B}_N\bar{U} = \mu_N - \bar{A}_N\mu_0 - \bar{C}_N \bar{V}.
		\end{align}	
		Since $\mu_0$ and $\mu_N$ are known, and since rank$[\bar{B}_N] = n$, there exists a solution for $\bar{U}$ for every $\bar{V} \in \mathbb{R}^{N\ell}$. 
		Hence, the controller's feasible set is non-empty.
	\end{proof}
	
	\vspace{0.1in}
	
	From the above lemma, it is obvious that the controller can meet the constraint (\ref{eq:org_meancont}), if the condition rank$[\bar{B}_N] = n$ is satisfied.
	In the upper game, the controller input is obtained first and the stopper best responds accordingly. 
	The terminal condition (\ref{eq:org_meancont}) depends on the stopper input. 
	Note that it is assumed that the stopper is indifferent to this constraint, and in this regard, the sufficient condition for which the controller's terminal constraint is met is derived in Lemma \ref{lemma:cmsg_cond} below.
	
	From equations (\ref{eq:mean_grad1}), (\ref{eq:mean_grad2}), the players' best responses as a function of their opponent response can be obtained as
	\begin{subequations}
		\begin{align}
		\bar{U} &= -(\mathcal{B}^\top\bar{Q}\mathcal{B} + \bar{R})^{-1} (\mathcal{B}^\top\bar{Q}\mathcal{C}\bar{V} + \mathcal{B}^\top\bar{Q}\mathcal{A}\mu_0), \\
		\bar{V} &= -(\mathcal{C}^\top\bar{Q}\mathcal{C} - \bar{S})^{-1} (\mathcal{C}^\top\bar{Q}\mathcal{B}\bar{U} + \mathcal{C}^\top\bar{Q}\mathcal{A}\mu_0). \label{eq:vbest}
		\end{align}
	\end{subequations} 
	In the upper game, where the controller plays first, the stopper input as a function of $\bar{U}$ is given by (\ref{eq:vbest}).
	Given the stopper input (as per (\ref{eq:vbest})), from the constraint (\ref{eq:org_meancont}), it follows that
	\begin{align}
	\mu_N &= \bar{A}_N\mu_0 + \bar{B}_N \bar{U}  + \bar{C}_N \big(-(\mathcal{C}^\top\bar{Q}\mathcal{C} - \bar{S})^{-1} (\mathcal{C}^\top\bar{Q}\mathcal{B}\bar{U} + \mathcal{C}^\top\bar{Q}\mathcal{A}\mu_0)\big) \nonumber \\
	&= \big(\bar{A}_N - \bar{C}_N(\mathcal{C}^\top\bar{Q}\mathcal{C} - \bar{S})^{-1}\mathcal{C}^\top\bar{Q}\mathcal{A}\big) \mu_0  + \big(\bar{B}_N -\bar{C}_N(\mathcal{C}^\top\bar{Q}\mathcal{C} - \bar{S})^{-1}\mathcal{C}^\top\bar{Q}\mathcal{B} \big) \bar{U}. \label{eq:pur_DS}
	\end{align}
	
	For the sake of brevity, let $\mathcal{G} = \bar{B}_N - \bar{C}_N(\mathcal{C}^\top\bar{Q}\mathcal{C} - \bar{S})^{-1}\mathcal{C}^\top\bar{Q}\mathcal{B}$.
	
	\vspace{0.1in}
	
	\begin{lemma} \label{lemma:cmsg_cond}
		Given the CMSG (\ref{eq:const_mean_game}), in the associated upper game, the constraint (\ref{eq:org_meancont}) is satisfied if and only if
		\begin{align}
		&\text{rank}\left[\mathcal{G} \quad \mu_N - \left(\bar{A}_N - \bar{C}_N (\mathcal{C}^\top\bar{Q}\mathcal{C} - \bar{S})^{-1}\mathcal{C}^\top\bar{Q}\mathcal{A}\right) \mu_0\right] = \text{rank}\left[\mathcal{G}\right]. \label{eq:rank_cond}
		\end{align}
	\end{lemma}
	
	\vspace{0.1in}
	
	\begin{proof}
		The condition (\ref{eq:rank_cond}) suggests that the system of linear equations, obtained from (\ref{eq:pur_DS}),
		\begin{align}
		\mathcal{G} \bar{U} = \mu_N - \left(\bar{A}_N - \bar{C}_N (\mathcal{C}^\top\bar{Q}\mathcal{C} - \bar{S})^{-1}\mathcal{C}^\top\bar{Q}\mathcal{A}\right) \mu_0,
		\end{align}
		has a solution for $\bar{U}$. Therefore, there always exists a constrained upper value for the CMSG, and the controller can drive the state to a given $\mu_N$ at the $N^{th}$ time-step.
	\end{proof}
	
	\vspace{0.1in}
	
	Note that the matrix $\mathcal{G}$ can be treated as a \emph{relative controllability matrix}, similar to the one introduced in Ref. \cite{Ho1968} for continuous systems.
	The optimal control sequences $\bar{U}_*$ and $\bar{V_*}$ that solve the upper game can be found as follows. 
	From (\ref{eq:vbest}), the upper game can be expressed in terms of the following minimization problem.
	\begin{align}
	\begin{cases}
	\underset{\bar{U} \in \mathbb{R}^{Nm}}{\min} \bar{X}^\top \bar{Q}\bar{X} + \bar{U}^\top \bar{R}\bar{U} - \bar{V}^\top \bar{S}\bar{V},\\
	\text{subject to } \mu_N = \bar{A}_N\mu_0 + \bar{B}_N \bar{U} + \bar{C}_N \bar{V}, \label{eq:mean_min}
	\end{cases}
	\end{align}
	where $\bar{X} = \mathcal{A}\mu_0 + \mathcal{B}\bar{U} + \mathcal{C}\bar{V}$, and $\bar{V} = -(\mathcal{C}^\top\bar{Q}\mathcal{C} - \bar{S})^{-1} (\mathcal{C}^\top\bar{Q}\mathcal{B}\bar{U} + \mathcal{C}^\top\bar{Q}\mathcal{A}\mu_0)$.
	
	\vspace{0.1in}
	
	\begin{proposition}
		Under the assumption
		\begin{align}
		\text{rank}~\mathcal{G} = n, \label{eq:rel_contr_cond}
		\end{align}
		the optimal control sequence $\bar{U}_*$ that solves the minimization problem in (\ref{eq:mean_min}) is given by 
		\begin{align}
		\bar{U}_* = \mathcal{R}^{-1}\left(\mathcal{M} + \mathcal{G}^\top\lambda/2\right), \label{eq:ubest_minmax}
		\end{align}
		where
		\begin{subequations}
			\begin{align}
			\mathcal{R} &= \bar{R} + \mathcal{B}^\top\bar{Q}\mathcal{B} - \mathcal{B}^\top\bar{Q}\mathcal{C}(\mathcal{C}^\top\bar{Q}\mathcal{C} - \bar{S})^{-1}\mathcal{C}^\top\bar{Q}\mathcal{B}, \label{eq:scriptR} \\
			\mathcal{M} &= \big(\mathcal{B}^\top\bar{Q}\mathcal{C}(\mathcal{C}^\top\bar{Q}\mathcal{C} - \bar{S})^{-1}\mathcal{C}^\top - \mathcal{B}^\top\big)\bar{Q}\mathcal{A}\mu_0, \label{eq:scriptM}\\
			\lambda &= 2\big(\mathcal{G}\mathcal{R}^{-1}\mathcal{G}^\top\big)^{-1} \big(\mu_N - \bar{A}_N\mu_0 + \bar{C}_N(\mathcal{C}^\top\bar{Q}\mathcal{C} - \bar{S})^{-1}\mathcal{C}^\top\bar{Q}\mathcal{A}\mu_0 - \mathcal{G}\mathcal{R}^{-1}\mathcal{M}\big).
			\end{align}
		\end{subequations}
	\end{proposition}
	
	
	\begin{proof}The Lagrangian for the constrained minimization problem (\ref{eq:mean_min}) can be written as
		\begin{align}
		\mathcal{L}(\bar{U},\lambda) &= \bar{X}^\top \bar{Q}\bar{X} + \bar{U}^\top \bar{R}\bar{U} - \bar{V}^\top \bar{S}\bar{V} + \lambda^\top(\mu_N - \bar{A}_N\mu_0 - \bar{B}_N \bar{U} - \bar{C}_N \bar{V}) \nonumber \\
		&= (\mathcal{A}\mu_0 + \mathcal{B}\bar{U} + \mathcal{C}\bar{V})^\top \bar{Q}(\mathcal{A}\mu_0 + \mathcal{B}\bar{U} + \mathcal{C}\bar{V}) + \bar{U}^\top \bar{R}\bar{U} - \bar{V}^\top \bar{S}\bar{V} + \lambda^\top(\mu_N - \bar{A}_N\mu_0 - \bar{B}_N \bar{U} - \bar{C}_N \bar{V}),
		\end{align}	
		where $\lambda \in \mathbb{R}^{n}$. The first-order optimality condition yields
		\begin{align}
		\nabla_{\bar{U}} \mathcal{L} &= 2(\mathcal{A}\mu_0 + \mathcal{B}\bar{U} + \mathcal{C}\bar{V})^\top\bar{Q}\left(\mathcal{B} + \mathcal{C}\frac{\partial \bar{V}}{\partial \bar{U}}\right) + 2\bar{U}^\top \bar{R}  - 2\bar{V}^\top\bar{S}\frac{\partial \bar{V}}{\partial \bar{U}} + \lambda^\top\left( - \bar{B}_N - \bar{C}_N \frac{\partial \bar{V}}{\partial \bar{U}}\right) = 0, 
		\end{align} 
		and (\ref{eq:ubest_minmax}) follows from the fact that $\dfrac{\partial \bar{V}}{\partial \bar{U}} = - (\mathcal{C}^\top\bar{Q}\mathcal{C} - \bar{S})^{-1}\mathcal{C}^\top\bar{Q}\mathcal{B}$ (obtained using (\ref{eq:vbest})), and from the second-order optimality condition
		\begin{align}
		\frac{\nabla_{\bar{U}\bar{U}} \mathcal{L}}{2} &= \left(\mathcal{B} + \mathcal{C}\frac{\partial \bar{V}}{\partial \bar{U}}\right)^\top\bar{Q}\left(\mathcal{B} + \mathcal{C}\frac{\partial \bar{V}}{\partial \bar{U}}\right) + \bar{R} -\frac{\partial \bar{V}}{\partial \bar{U}}^\top\bar{S}\frac{\partial \bar{V}}{\partial \bar{U}} \nonumber \\
		&= \left(\bar{R} + \mathcal{B}^\top\bar{Q}\mathcal{B} - \mathcal{B}^\top\bar{Q}\mathcal{C}(\mathcal{C}^\top\bar{Q}\mathcal{C} - \bar{S})^{-1}\mathcal{C}^\top\bar{Q}\mathcal{B}\right) \nonumber\\
		&= \mathcal{R} \succ 0
		\end{align}
		The Lagrange multiplier $\lambda$ can be found by substituting (\ref{eq:ubest_minmax}) along with (\ref{eq:scriptR}) and (\ref{eq:scriptM}) into the terminal constraint, obtaining 
		\begin{align}
		&\big(\mathcal{G}\mathcal{R}^{-1}\mathcal{G}^\top\big) \lambda = 2\big(\mu_N - \bar{A}_N\mu_0 +  \bar{C}_N(\mathcal{C}^\top\bar{Q}\mathcal{C} - \bar{S})^{-1}\mathcal{C}^\top\bar{Q}\mathcal{A}\mu_0 - \mathcal{G}\mathcal{R}^{-1}\mathcal{M}\big)
		\end{align}
		Note that since $\mathcal{R}$ is invertible and $\mathcal{G}$ has full row rank, $\mathcal{G}\mathcal{R}^{-1}\mathcal{G}^\top$ is invertible.
	\end{proof}


\section{Covariance Steering Game}
	\label{sec:covariance}
	
	The methodology to solve the UCSG and the CCSG is presented in this section. 
	Assuming a linear feedback control structure for steering the covariance, we express $\tilde{U}$ and $\tilde{V}$ as 
	\begin{align}
	\tilde{u}_k = K_ky_k, \quad \tilde{v}_k = L_ky_k, \label{eq:UV_tilde_y}
	\end{align}
	where $K_k \in \mathbb{R}^{m \times n}$, $L_k \in \mathbb{R}^{\ell \times n}$,
	\begin{subequations}
		\begin{align}
		y_{k+1} &= A_ky_k + D_kw_k,\\
		y_0 &= x_0 - \mu_0,
		\end{align}
	\end{subequations}
	and $y_k \in \mathbb{R}^{n}$. 
	Note that $\mathbb{E}[y_0] = 0$ and $\mathbb{E}[y_0y_0^\top] = \Sigma_0$.
	Further, it can be obtained that 
	\begin{align}
	Y = \mathcal{A}y_0 + \mathcal{D}W, \label{eq:Y_dyn}
	\end{align}
	where $Y = [y_0^\top,\dots,y_N^\top]^\top \in \mathbb{R}^{(N+1)n}$, using the matrices introduced in Section \ref{sec:prelims}.
	Therefore, $\tilde{X}$ in (\ref{eq:Xtilde}) can be rewritten as
	\begin{align}
	\tilde{X} = (I + \mathcal{B}K + \mathcal{C}L)(\mathcal{A}y_0 + \mathcal{D}W). \label{eq:new_Xtilde}
	\end{align}
	where,
	\begin{subequations} \label{eq:gainsKL}
		\begin{align}
		K = \left[\begin{array}{ccccc}
		K_0 & 0 & \dots & 0 & 0\\
		0 & K_1 & \dots & 0 & 0\\
		\vdots & \vdots & \ddots & \vdots & \vdots\\
		0 & 0 & \dots & K_{N-1} & 0
		\end{array}\right], \label{K_structure}\\
		L = \left[\begin{array}{ccccc}
		L_0 & 0 & \dots & 0 & 0\\
		0 & L_1 & \dots & 0 & 0\\
		\vdots & \vdots & \ddots & \vdots & \vdots\\
		0 & 0 & \dots & L_{N-1} & 0
		\end{array}\right], \label{L_structure}
		\end{align}
	\end{subequations}
	are the controller and the stopper gain matrices, respectively.
	Here $K \in \mathbb{K}$ and $L \in \mathbb{L}$, where $\mathbb{K}$ is the set of $Nm\times (N+1)n$ matrices that have the structure shown in (\ref{K_structure}), and similarly, $\mathbb{L}$ is the set of $N\ell \times (N+1)n$ matrices that have the structure shown in (\ref{L_structure}).
	From (\ref{eq:UV_tilde_y}), (\ref{eq:Y_dyn}), and (\ref{eq:new_Xtilde}), we have $\mathbb{E}[\tilde{X}\tilde{X}^\top] = (I + \mathcal{B}K + \mathcal{C}L)\Sigma_s(I + \mathcal{B}K + \mathcal{C}L)^\top$, $\mathbb{E}[\tilde{U}\tilde{U}^\top] = K\Sigma_sK^\top$, $\mathbb{E}[\tilde{V}\tilde{V}^\top] = L\Sigma_sL^\top$, where $\Sigma_s = \mathcal{A}\Sigma_0\mathcal{A}^\top + \mathcal{D}\mathcal{D}^\top$. 
	Therefore, the cost function (\ref{eq:cov_cost}) can be converted to the following quadratic form in terms of $K$ and $L$:
	\begin{align}
	J_\Sigma(K,L) &= \text{tr}(((I + \mathcal{B}K + \mathcal{C}L)^\top\bar{Q}(I + \mathcal{B}K + \mathcal{C}L) + K^\top\bar{R}K - L^\top\bar{S}L)\Sigma_s), \label{eq:new_cov_cost}
	\end{align}
	and the terminal constraint (\ref{eq:org_covcont}) can be rewritten as
	\begin{align}
	\Sigma_N = E_N (I + \mathcal{B}K + \mathcal{C}L)\Sigma_s(I + \mathcal{B}K + \mathcal{C}L)^\top E_N^\top. \label{eq:new_cov_const}
	\end{align}
	
	For the sake of analysis, we introduce the set $\tilde{\mathcal{R}} $.
Given stopper gain $L$ in CCSG, let $\mathcal{K}(L)$ denotes the set of gains $K \in \mathbb{K}$ for which the controller satisfies the constraint in (\ref{eq:new_cov_const}), and let $\tilde{\mathcal{R}} \triangleq \bigcup_{L \in \mathbb{L}}~\mathcal{K}(L) \subseteq \mathbb{K}$.
	
	We first analyze the UCSG.
	Since the gain matrices $K$ and $L$ have constraints on their structure with zeros, as shown in (\ref{eq:gainsKL}), with a slight abuse of notation, the Lagrangian can be written as
	\begin{align}
	\mathcal{L}(K,L,\Theta,\Xi) &=  \text{tr}(((I + \mathcal{B}K + \mathcal{C}L)^\top\bar{Q}(I + \mathcal{B}K + \mathcal{C}L) + K^\top\bar{R}K - L^\top\bar{S}L)\Sigma_s)/2  \nonumber \\
	&~~~+ \sum_{i = 1}^{Nm}\sum_{j \in \mathscr{J}_k(i)} \theta_{ij}e_i^\top K e_j + \sum_{i = 1}^{N\ell}\sum_{j \in \mathscr{J}_l(i)} \xi_{ij}e_i^\top L e_j, \label{eq:cov_lagra}
	\end{align}
	where the functions $\mathscr{J}_k(.)$ and $\mathscr{J}_l(.)$ map each row number to the set of columns in which the gains $K$ and $L$, respectively, have zero  elements. The matrices $\Theta \in \mathbb{R}^{Nm \times (N+1)n}$ and $\Xi \in \mathbb{R}^{N\ell \times (N+1)n}$ are Lagrange multipliers of sizes equal to $K$ and $L$, respectively. 
	Note that the blocks in $\Theta$ and $\Xi$ (corresponding to $K_k$ and $L_k$) are zeros, and $\theta_{ij}$ and $\xi_{ij}$ are the non-zero elements of these matrices.
	The first-order necessary conditions for the existence of a saddle point can be obtained by taking derivatives of the Lagrangian in (\ref{eq:cov_lagra}) with respect to $K$ and $L$ as
	\begin{subequations} \label{eq:cov_grads}
		\begin{align}
		\nabla_{K} \mathcal{L} &= \left[\mathcal{B}^\top\bar{Q} + \bar{R}K + \mathcal{B}^\top\bar{Q}\mathcal{B}K + \mathcal{B}^\top\bar{Q}\mathcal{C}L\right] \Sigma_s + \Theta = 0, \label{eq:cov_grad1}\\
		\nabla_{L} \mathcal{L} &= \left[\mathcal{C}^\top \bar{Q} - \bar{S}L + \mathcal{C}^\top\bar{Q}\mathcal{B}K +\mathcal{C}^\top\bar{Q}\mathcal{C}L\right]\Sigma_s + \Xi = 0. \label{eq:cov_grad2}
		\end{align}
	\end{subequations}
	The candidate solutions for the UCSG can be obtained by solving the linear system of equations given in (\ref{eq:cov_grads}).
	Since the gradients are linear, the second-order sufficient conditions, using the bordered Hessians, can be invoked to find the saddle points among the candidate solutions numerically \cite{magnus_matrix1999}. 
	
	Next, we analyze the CCSG. A result similar to the one proposed for the CMSG (Lemma \ref{lemma:mean_const}) follows for the CCSG and is given below. The proof is omitted as it is similar to the one given for Lemma \ref{lemma:mean_const}.
	
	\vspace{0.1in}
	
	\begin{lemma}
		Assuming that the UCSG with payoff function (\ref{eq:new_cov_cost}) has a saddle point equilibrium, then the CCSG (\ref{eq:org_covcont}), with the terminal constraint (\ref{eq:new_cov_const}) imposed only for the controller, satisfies
		\begin{align}
		\underset{K \in \mathbb{K}}{\inf}~\underset{L \in \mathbb{L}}{\sup}~J_\Sigma(K,L) \leq\underset{L \in \mathbb{L}}{\sup}~\underset{K \in \tilde{\mathcal{R}}}{\inf}~J_\Sigma(K,L).
		\end{align}
	\end{lemma}
	
	\vspace{0.1in}
	
	Similarly, in the CCSG, a pure-strategy equilibrium need not exist.
	To this end, consider a simple Jacobi procedure given in Algorithm \ref{algo:J_Nash} to arrive at an equilibrium solution, assuming one exists.
	For Algorithm \ref{algo:J_Nash} to converge to an equilibrium solution for any $K_0$, $L_0$, the solution has to be a stable one \cite{li1987distributed}. 
	The conditions for the existence of a stable equilibrium for the case where the cost is convex in $K$ and concave in $L$ can be found in Ref. \cite{li1987distributed}.
	
	\begin{algorithm}
		\caption{Jacobi procedure to obtain saddle points}\label{algo:J_Nash}
		\begin{algorithmic}[1]
			\Procedure{Jacobi}{$K_0$,$L_0$}
			\For {i = 0,1,2,\dots}
			\State $L_{i+1} := \underset{L \in \mathbb{L}}{\arg\max}~J_\Sigma(K_i,L)$
			\State $K_{i+1} := \underset{K \in \mathcal{K}(L_{i})}{\arg\min}~J_\Sigma(K,L_{i})$
			\EndFor
			\State \Return $K_{i+1}$, $L_{i+1}$
			\EndProcedure
		\end{algorithmic}
	\end{algorithm}
	
	Subsequently, under the assumptions that $\Sigma_s \otimes (\mathcal{B}^\top\bar{Q}\mathcal{B} + \bar{R}) \succ 0$ (convex in $K$) and $\Sigma_s \otimes (\mathcal{C}^\top\bar{Q}\mathcal{C} - \bar{S}) \prec 0$ (concave in $L$), we can formulate the successive minimization and maximization problems as convex programming problems by relaxing the equality constraint in (\ref{eq:new_cov_const}) to an inequality constraint, 
	\begin{align}
	\Sigma_N \succeq E_N (I + \mathcal{B}K + \mathcal{C}L)\Sigma_s(I + \mathcal{B}K + \mathcal{C}L)^\top E_N^\top. \label{eq:org_cov_constraint}
	\end{align} 
	
	\vspace{0.1in}
	
	\begin{lemma}
		Assuming $\Sigma_N \succ 0$, the inequality constraint (\ref{eq:org_cov_constraint}) can be expressed as
		\begin{align}
		\|\Sigma_N^{-1/2}E_N(I + \mathcal{B}K + \mathcal{C}L)\Sigma_s^{1/2}\|_2 - 1 \leq 0. \label{eq:mod_cov_constr}
		\end{align}
	\end{lemma}

	\vspace{0.1in}

	\begin{proof}
		Since assumption $\Sigma_N \succ 0$, (\ref{eq:org_cov_constraint}) can be rewritten as
		\begin{align*}
		I - \Sigma_N^{-1/2} E_N (I + \mathcal{B}K + \mathcal{C}L)\Sigma_s (I + \mathcal{B}K + \mathcal{C}L)^\top E_N^\top \Sigma_N^{-1/2} \succeq 0.
		\end{align*}
		As it is symmetric, the matrix $\Sigma_N^{-1/2} E_N (I + \mathcal{B}K + \mathcal{C}L)\Sigma_s(I + \mathcal{B}K + \mathcal{C}L)^\top E_N^\top \Sigma_N^{-1/2}$ is diagonalizable via an orthogonal matrix $T \in \mathbb{R}^{n\times n}$ as
		\begin{align}
		T (I_n - \mathrm{diag}(\lambda_1,\dots,\lambda_n))S^\top \succeq 0, \label{eq:eig_value}
		\end{align}
		where $\lambda_1,\dots,\lambda_n$ are its eigenvalues.
		From (\ref{eq:eig_value}), we have
		\begin{align}
		1 - \lambda_{max}\big(\Sigma_N^{-1/2} &E_N (I + \mathcal{B}K + \mathcal{C}L)\Sigma_s (I + \mathcal{B}K + \mathcal{C}L)^\top E_N^\top \Sigma_N^{-1/2}\big) \geq 0.\\
		\implies 1 - \|\Sigma_N^{-1/2}&E_N(I + \mathcal{B}K + \mathcal{C}L)\Sigma_s^{1/2}\|_2 \geq 0.
		\end{align}
		Hence proved.
	\end{proof}


\section{Numerical Simulations}
	\label{sec:sims}
	
	As mentioned earlier, in the lower game of the mean steering case, the controller has an advantage to drive the distribution to a given terminal Gaussian, assuming the system is controllable. 
	A more challenging case is that of the upper game, where the controller has to ensure that the terminal constraint (\ref{eq:org_meancont}) is met while choosing its input first.
	In this section, we first present test examples for the upper game of the CMSG with linear time-invariant systems, and then analyze the missile end-game guidance problem.
	For the covariance steering part, YALMIP \cite{Lofberg2004} in conjunction with MOSEK \cite{mosek2017} was used to solve the successive convex optimization problems in the Jacobi procedure.
	The convergence criterion for the iterative method is $\epsilon_k, \epsilon_\ell \leq \epsilon$, where $\epsilon_k = \|K_{i+1} - K_i\|$ and $\epsilon_\ell = \|L_{i+1} - L_i\|$.
	
	\begin{figure}[htb]
		\centering
		\includegraphics[width=0.4\textwidth]{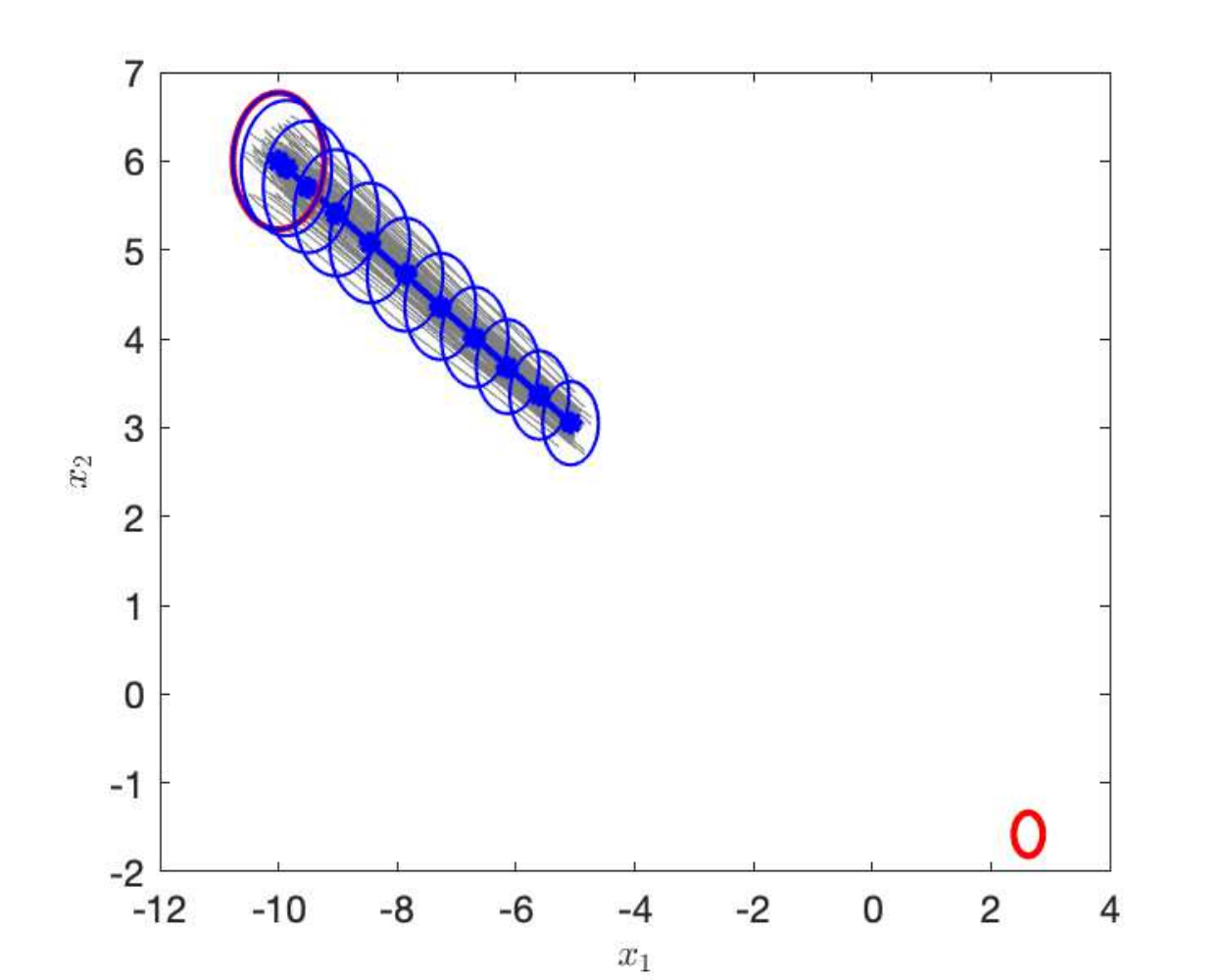}
		\caption{Unconstrained mean and covariance steering}
		\label{fig:unconstrained}
	\end{figure}
	
	\subsection{Test Example} 
	
	Consider the linear system
	\begin{align}
	z_{k+1} = Az_k + Bu_k + Cv_k + Dw_k \label{eq:discrete_dyn}
	\end{align}
	where $z_k = [x_1,x_2,x_3,x_4]^\top \in \mathbb{R}^4$, $u_k, v_k \in \mathbb{R}^2$, $w_k \in \mathbb{R}^4$,
	\begin{align}
	A = \left[\begin{array}{cccc}
	1 & 0 & \Delta t & 0\\
	0 & 1 & 0 & \Delta t\\
	0 & 0 & 1 & 0\\
	0 & 0 & 0 & 1
	\end{array}\right], \quad
	B = \left[\begin{array}{cc}
	\Delta t^2 & 0 \\
	0 & \Delta t^2 \\
	\Delta t & 0 \\
	0 & \Delta t
	\end{array}\right],
	\end{align}
	$C = -B$, and $D = 0.01I_4$.
	Note that $x_1$, $x_2$ can be understood as relative coordinates, and $x_3$, $x_4$ are the relative velocities along the $x_1$ and $x_2$ axes, respectively, with $\Delta t = 0.2$ being the time-step size.
	Finally, $u_k$ and $v_k$ are the accelerations of the pursuer (controller) and the evader (stopper), respectively.
	
\begin{figure}[htb!]
    \centering
    \subfigure[Sampled Trajectories]{\includegraphics[width = 0.4\textwidth]{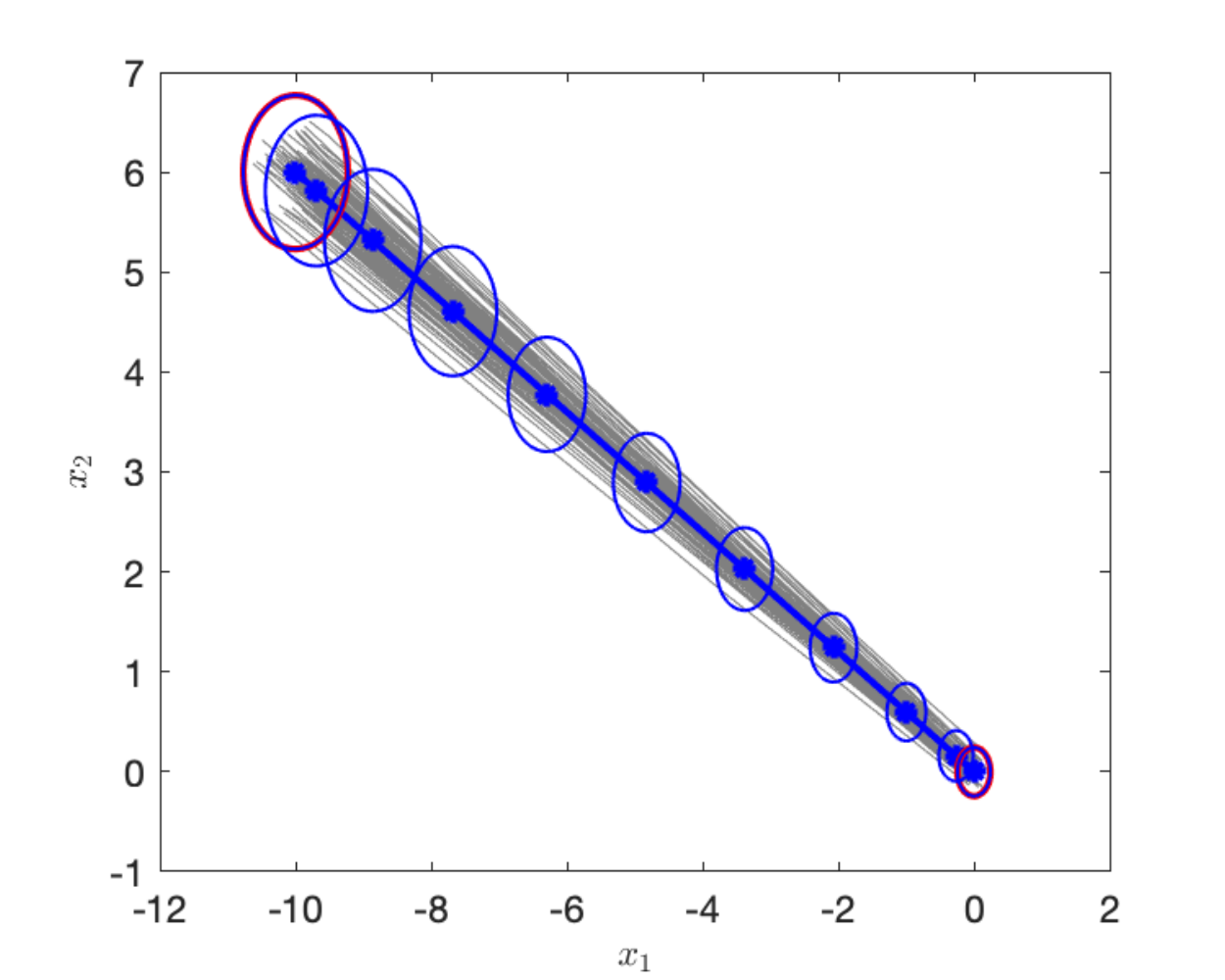}}
    \subfigure[Convergence of the Jacobi procedure]{\includegraphics[width = 0.4\textwidth]{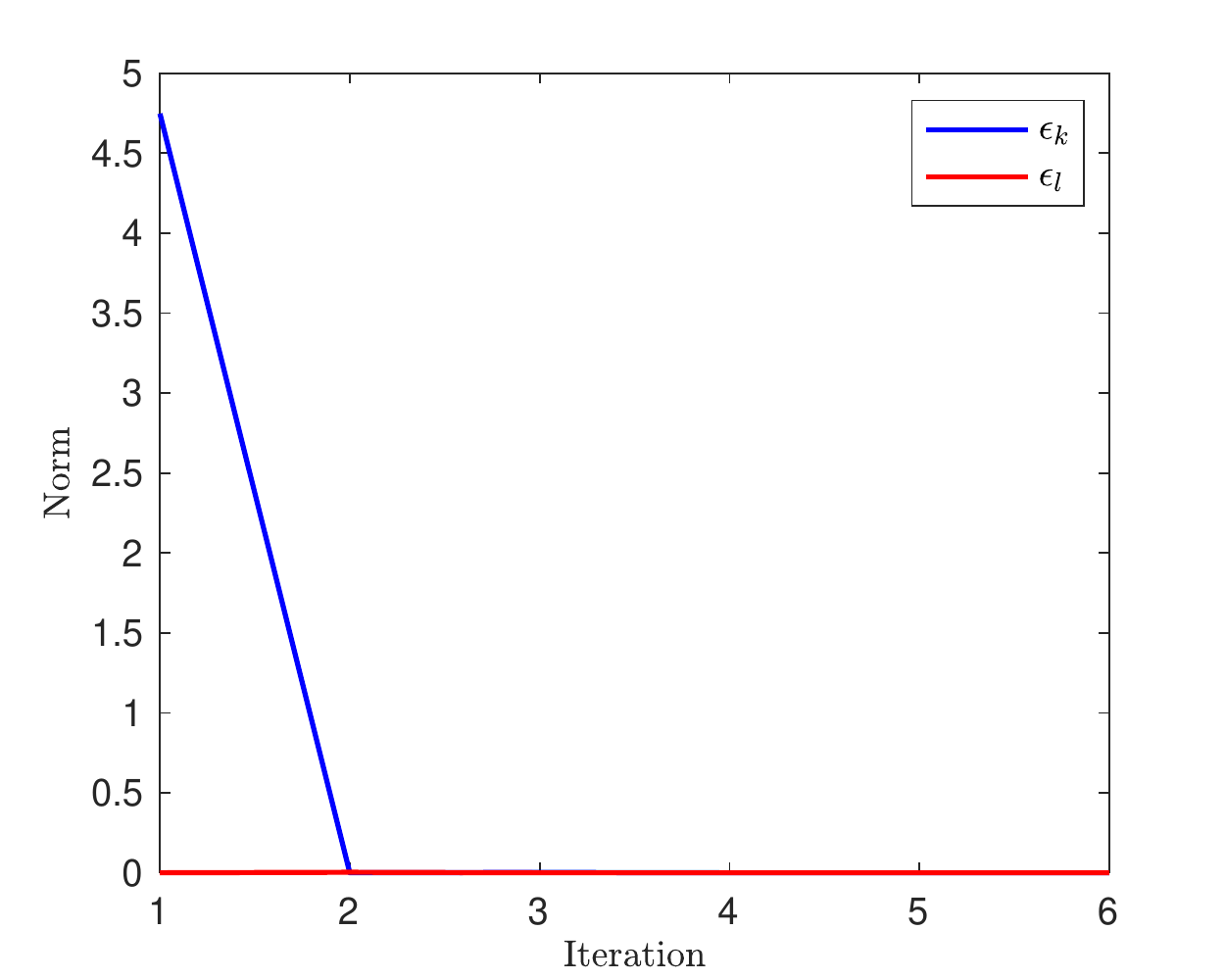}}
    \caption{Constrained mean and covariance steering: A case where the covariance condition is met by the controller.}
    \label{fig:infsup_work}
\end{figure}

	The initial condition is chosen to be 
	$\mu_0 = [-10,6,0,0]^\top$, $\Sigma_0 =$ blkdiag(0.05, 0.05, 0.01, 0.01),
	and the terminal constraint is
	$\mu_N = [0,0,0,0]^\top$, $\Sigma_N =$ blkdiag(0.005, 0.005, 0.001, 0.001).
	The time horizon is fixed at $N = 10$, and the cost matrices are $Q_k = I_4$, and $R_k = I_2$ $S_k = 100I_2$, for all $k \geq 1$.
	The instance is first analyzed without the terminal constraint, and the solution to the unconstrained game is obtained. 
	The UCSG is solved using the Jacobi procedure illustrated in Ref. \cite{li1987distributed}.
	Fig. \ref{fig:unconstrained} presents the solution to the unconstrained game for the given initial condition.
	The red ellipses in Fig. \ref{fig:unconstrained} denote the $3\sigma$ error of the initial and the desired terminal state distributions of $x_1$ and $x_2$ coordinates. 
	The blue solid line illustrates the mean trajectory, and the blue ellipses illustrate the covariance evolution over the time horizon.
	The gray lines are the trajectories simulated for 100 different initial conditions that are sampled from $\mathcal{N}(\mu_0,\Sigma_0)$.
	It can be observed that the mean and the covariance trajectories do not meet the controller's terminal conditions.
	For the CDG, the relative controllability matrix is found to have full row rank, and therefore the mean can be driven to the specified terminal value.
	Also, the covariance steering problem is feasible with $\epsilon = 10^{-5}$, and the result is illustrated in Fig. \ref{fig:infsup_work}. 
	The convergence of the Jacobi procedure (Algorithm \ref{algo:J_Nash}) can be observed in Fig. \ref{fig:infsup_work}(b).
	From Fig. \ref{fig:infsup_work}(a), it can be observed that the covariance constraint is satisfied.
	
	\begin{figure}[htb]
		\centering
		\includegraphics[width=0.4\textwidth]{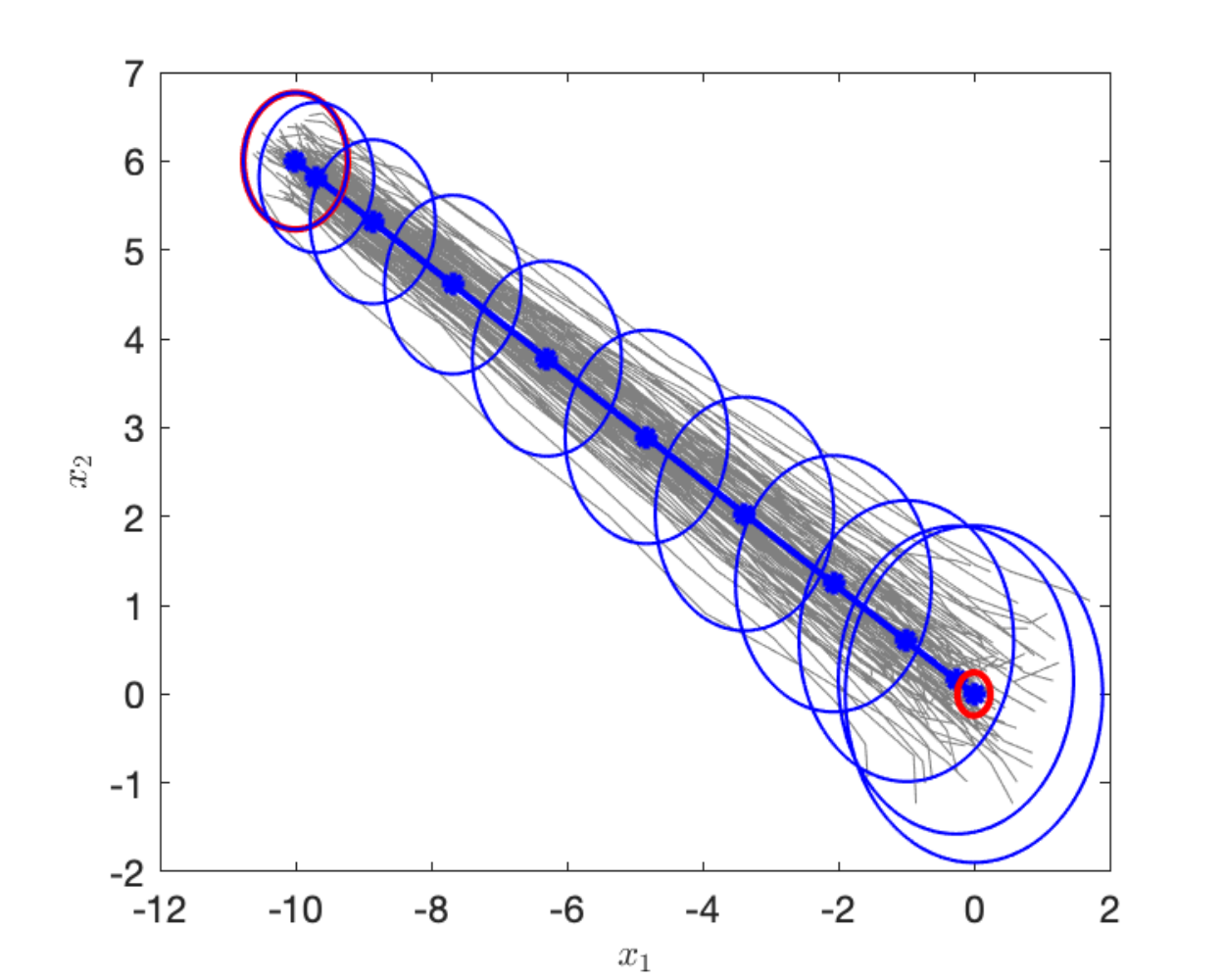}
		\caption{A case where the terminal covariance constraint is not met.}
		\label{fig:cov_notconvg}
	\end{figure}
	
	Fig. \ref{fig:cov_notconvg} illustrates the case where $D = 0.1I_4$, while the rest of the values are kept unchanged.
	Since changing the matrix $D$ does not change the behavior of the mean, in this case, the mean converges to the specified terminal value.
	However, the covariance constraint cannot be achieved in this case and from Fig. \ref{fig:cov_notconvg}, it can be observed that the covariance ellipse grows with time.
	The result in Fig. \ref{fig:cov_notconvg} is for the set of optimal gains $(K^*,L^*)$, obtained by  minimizing the cost (\ref{eq:cov_cost}) subject to the constraint (\ref{eq:cov_grad2}), since the constraint (\ref{eq:mod_cov_constr}) cannot be met.
	
	\subsection{Missile Endgame Guidance} 
	
	Fig. \ref{fig:missile_scenario} presents a schematic of the interception geometry during a missile engagement scenario. During the endgame, the relative dynamics can be linearized along the initial line of sight while assuming a constant closing speed $V_c = V_p + V_e$, where $V_p$ and $V_e$ are the constant speeds of the missile and of the target, respectively.
	Trajectory linearization is well established for ballistic missile defense where the endgame is over a short duration, and begins with near ``head-on" initial conditions. Now, consider the linearized dynamics of a missile during the end-game in continuous time \cite{turetsky2003},
	\begin{align}
	\dot{z} = A_cz + B_cu + C_cv, \label{eq:cont_dyn}
	\end{align}
	where $z = [y,\dot{y},a_e,a_p]^\top$, $y$ is the target's relative distance to the missile normal to the reference line (initial line-of-sight), $\dot{y}$ is the relative speed, $a_e$ and $a_p$ are the lateral forces acting on the target and on the missile, respectively. 
	In (\ref{eq:cont_dyn}), $u$ and $v$ are the commanded lateral accelerations of the missile and of the target, respectively. The matrices in (\ref{eq:cont_dyn}) are given by
	\begin{subequations}
		\begin{align}
		A_c = \left[\begin{array}{cccc}
		0 & 1 & 0 & 0\\
		0 & 0 & 1 & -1\\
		0 & 0 & -1/\tau_e & 0\\
		0 & 0 & 0 & -1/\tau_p
		\end{array}\right], \quad
		B_c = \left[\begin{array}{c}
		0 \\
		0 \\
		0 \\
		1/\tau_p 
		\end{array}\right],\quad
		C_c = \left[\begin{array}{c}
		0  \\
		0 \\
		1/\tau_e \\
		0
		\end{array}\right],
		\end{align}
	\end{subequations}
	where $\tau_p$, and $\tau_e$ are model parameters \cite{turetsky2003}. The corresponding discrete matrices can be obtained from a given time-step size $\Delta t$, and the system evolution is expressed using (\ref{eq:discrete_dyn}) by assuming process noise in the system entering through the control channels.
	The system matrices for this example are given by $A = \exp(A_c \Delta t)$, $B = \int_{0}^{\Delta t} \exp(A_c\tau)B_c\diff{\tau}$, $C = \int_{0}^{\Delta t} \exp(A_c\tau)C_c\diff{\tau}$,
	\begin{align}
	D = \alpha \int_{0}^{\Delta t} \exp(A_c\tau)\diff{\tau} \left[\begin{array}{cc}
	0 & 0 \\
	0 & 0 \\
	0 & 1 \\
	1 & 0
	\end{array}\right],
	\end{align}
	where $\alpha$ is a constant. 
	\begin{figure}[htb]
		\centering
		\includegraphics[width=0.4\textwidth]{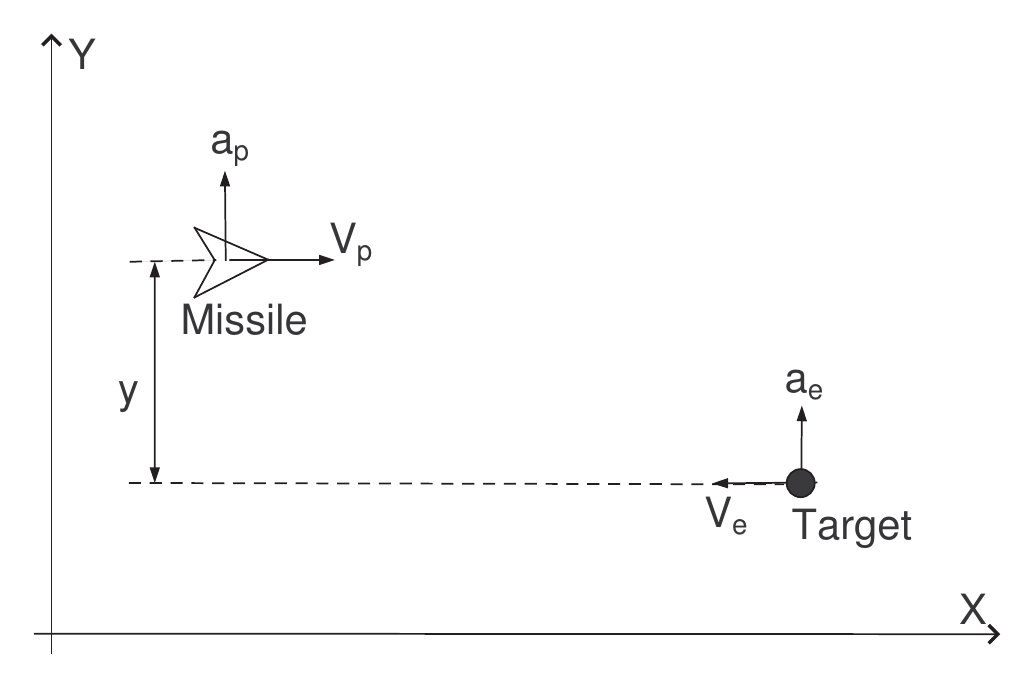}
		\caption{Planar end-game scenario}
		\label{fig:missile_scenario}
	\end{figure}
	
	For simulation purposes, the following parameters are chosen: $\Delta t = 0.1$, the initial separation along the line of sight $x_0 = 3500$, $\tau_e = 0.02$, $\tau_p = 0.01$. 
	We choose $V_p = 3000$, $V_e = 2000$, and therefore $t_f = x_0/(V_p+V_e) = 0.7$, $N = t_f/\Delta t = 7$.
	The initial conditions are  
	$\mu_0 = [0,350,0,0]^\top,~ \Sigma_0 = \text{blkdiag}(0.2,0.2,0.2,0.2),$
	and the terminal constraint is
	$\mu_N = [0,0,0.1,0.1]^\top,~ \Sigma_N = \text{blkdiag}(0.1,10,1,1).$
	Furthermore, $Q_k = 10^{-6} I_4$, $R_k = 10^2$, $S_k = 3\times 10^8$, for all $k \geq 1$.
	For this example, the relative controllability matrix is found to have full row rank, and the covariance steering problem is feasible. 
	The result is illustrated in Fig. \ref{fig:missile_example}, which shows the relative distance $y$ versus the time-step. 
	The errorbars (in red) indicate the $3\sigma$ error in $y$ at the initial and the final time-steps.
	From Fig. \ref{fig:missile_example}, it can be observed that the mean and covariance constraints are satisfied to successfully intercept the target during the end-game.
	
	\begin{figure}[htb]
		\centering
		\includegraphics[width=0.4\textwidth]{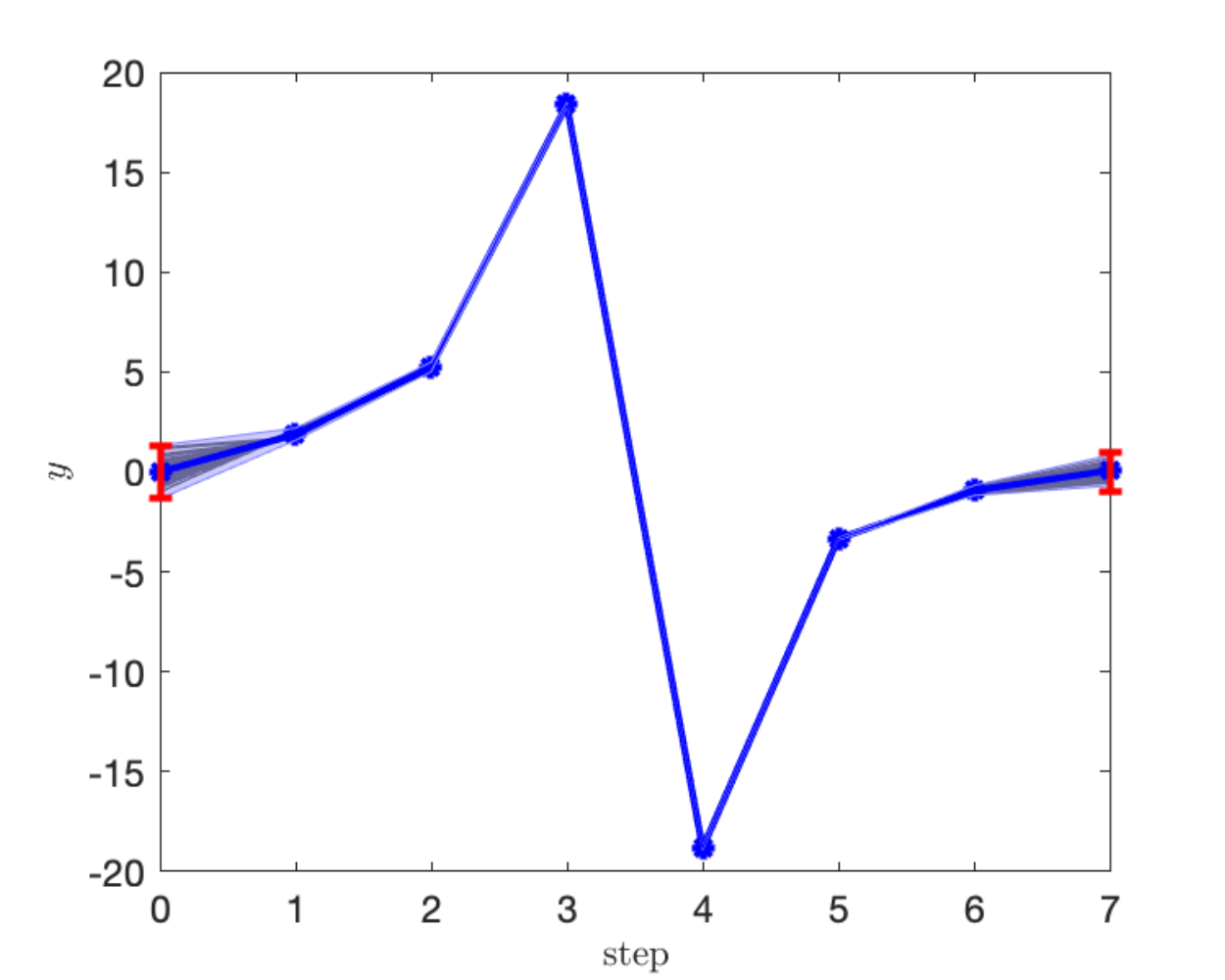}
		\caption{Mean and covariance steering for a missile engagement during the end-game}
		\label{fig:missile_example}
	\end{figure}
	
	Fig. \ref{fig:missile_traj} shows one of the many missile trajectories (generated for Fig. \ref{fig:missile_example}) relative to the target (red cross) in the $x-y$ plane. 
	Note that the motion along the $y$-axis is negligible compared to the missile's motion along the $x$-axis and consequently, a skewed aspect ratio has to be considered for Fig. \ref{fig:missile_example}. 
	
	\begin{figure}[htb]
		\centering
		\includegraphics[width=0.4\textwidth]{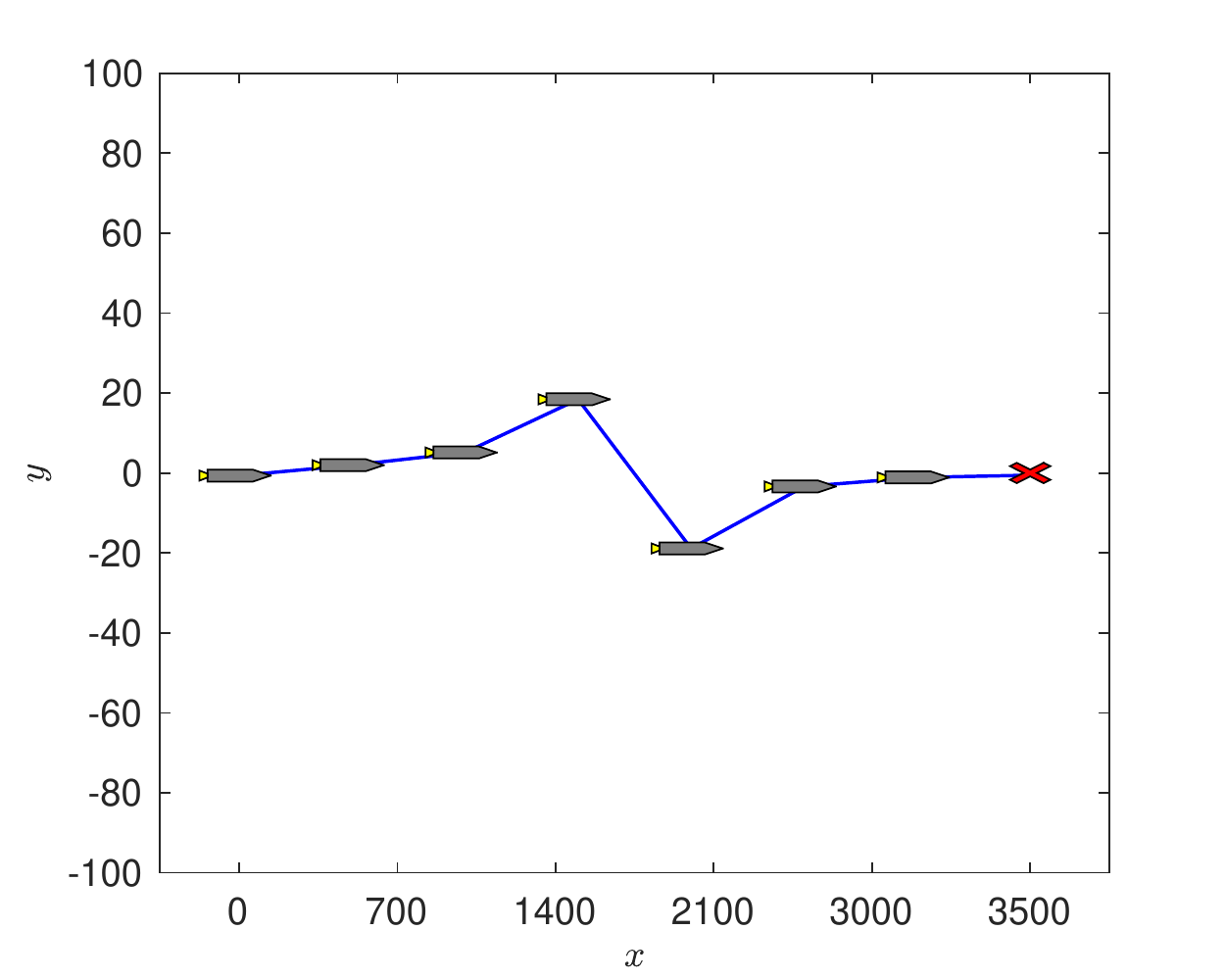}
		\caption{Missile's trajectory relative to the target}
		\label{fig:missile_traj}
	\end{figure}


\section{Conclusion} 
\label{sec:conclude}

This work addressed the problem of steering a Gaussian in adversarial scenarios using the theory of general constrained games. 
The problem is posed from a perspective of the player that desires to drive the distribution to a given terminal Gaussian while minimizing a quadratic cost. 
The player that tries to maximize the cost is assumed to be indifferent to the terminal constraint.
It is shown that the game need not have a saddle point equilibrium.
Subsequently, we obtained necessary conditions for the controller to drive the mean to the specified value in the upper game.
The covariance steering problem is solved numerically using the well-known Jacobi procedure.
The approach is illustrated via numerical examples.
Future work includes investigating the problem of missile and radar in the context of general constrained games in a discrete-time setting. 


\section*{Acknowledgment}                             
This work has been supported by NSF award CMMI-1662542.  

\end{document}